\documentclass[11pt,a4paper]{amsart}
\usepackage[utf8]{inputenc}
\usepackage{amsmath}
\usepackage{amsfonts}
\usepackage{amssymb}
\usepackage{mathrsfs}
\usepackage{amsxtra}
\usepackage{amsthm}
\theoremstyle{plain}
\newcommand{\C}{\mathbb{C}}
\newcommand{\F}{\mathscr{F}}
\newcommand{\Pn}{\mathbb{P}}

\newcommand{\sing}{\textsf{Sing}}

\newcommand{\LL}{{\mathscr L}}

\newcommand{\im}{\text{Im}}

\newcommand{\ov}[1]{\mbox{$\overline{#1}$}}

\newcommand{\R}{\mathbb{R}}
\newtheorem{theorem}{Theorem}[section]
\newtheorem{maintheorem}{Theorem}

\newtheorem{proposition}[theorem]{Proposition}

\newtheorem{maincorollary}{Corollary}

\theoremstyle{definition}
\newtheorem{definition}{Definition}[section]
\newtheorem{example}{Example}[section]

\title[Pull-back of singular Levi-flat hypersurfaces]{Pull-back of singular Levi-flat hypersurfaces}
\author{Andr\'es Beltr\'an}
\address[A. Beltr\'an]{Dpto. Ciencias - Secci\'on Matem\'aticas, Pontif\'icia Universidad Cat\'olica del Per\'u.}
\curraddr{Av. Universitaria 1801, San Miguel, Lima 32, Peru}
\email{abeltra@pucp.pe}
\author{Arturo Fern\'andez-P\'erez}
\address[A. Fern\'andez-P\'erez]{Departamento de Matem\'atica - ICEX, Universidade Federal de Minas Gerais, UFMG}
\curraddr{Av. Ant\^onio Carlos 6627, 31270-901, Belo Horizonte-MG, Brasil.}
\email{fernandez@ufmg.br}

\author{Hern\'an Neciosup}
\address[H. Neciosup]{Dpto. Ciencias - Secci\'on Matem\'aticas, Pontif\'icia Universidad Cat\'olica del Per\'u.}
\curraddr{Av. Universitaria 1801, San Miguel, Lima 32, Peru}
\email{hneciosup@pucp.pe}
\thanks{This work was supported by the Pontif\'icia Universidad Cat\'olica del Per\'u project VRI-DGI-2018-0024. The second author is partially supported by CNPq-Brazil Grant Number 302790/2019-5}
\begin{document}

\begin{abstract}
We study singular real analytic Levi-flat subsets invariant by singular holomorphic foliations in complex projective spaces. We give sufficient conditions for a real analytic Levi-flat subset to be the pull-back of a semianalytic Levi-flat hypersurface in a complex projective surface under a rational map or to be the pull-back of a real algebraic curve under a meromorphic function. 
In particular, we give an application to the case of a singular real analytic Levi-flat hypersurface. Our results improve previous ones due to Lebl and Bretas -- Fern\'andez-P\'erez -- Mol.   
\end{abstract}

\maketitle
\section{Introduction and statement of the results}
\par Let $M$ be a complex manifold of $\dim_\C M=N\geq 2$, a closed subset $H\subset M$ is a \textit{real analytic subvariety}  if for every $p\in H$, there are real analytic functions with real values $\varphi_1,\ldots,\varphi_k$ defined in a neighborhood $U\subset M$ of $p$, such that $H\cap U$ is equal to the set where all $\varphi_1,\ldots,\varphi_k$ vanish. A complex subvariety is precisely the same notion, considering holomorphic functions instead of real analytic functions.  We say that a real analytic subvariety $H$ is \textit{irreducible} if whenever we write $H=H_1\cup H_2$ for two subvarieties $H_1$ and $H_2$ of $M$, then either $H_1=H$ or $H_2=H$. If $H$ is irreducible, it has a well-defined dimension $\dim_\R H$. Let $H_{reg}$ denote its \textit{regular part}, i.e., the subset of points near which $H$ is a real analytic submanifold of dimension equal to $\dim_\R H$. A set is \textit{semianalytic} if it is locally constructed from real analytic sets by finite union, finite intersection, and complement. For a real analytic subvariety $H$, the set $\overline{H_{reg}}$ is a  semianalytic subset  where the closure is with the standard topology. In general, the inclusion $\overline{H_{reg}}\subset H$ is proper, which happens, for instance in the \textit{Whitney umbrella}. We really only study the set $\overline{H_{reg}}$, in this sense, we consider $\sing(H):=\overline{H_{reg}}\setminus H_{reg}$ as the  \textit{singular set} of $H$, this is not the usual definition of the singular set in the literature, see for instance \cite{singularlebl}. 
\par If $H\subset M$ is a real analytic hypersurface i.e., a real analytic subvariety of real codimension one, then for each $p\in H_{reg}$, there is a unique complex hyperplane $\LL_p\subset T_pH_{reg}$. This defines a real analytic distribution $p\mapsto \LL_p$ of complex hyperplanes in $TH_{reg}$. When this distribution is \textit{integrable} in the sense of Frobenius, we say that $H$ is \textit{Levi-flat}. Here, $H_{reg}$ is foliated by codimension one immersed complex submanifolds. This foliation, denoted by $\LL$, is known as \textit{Levi foliation}. According to Cartan \cite{cartan}, $\LL$ can be extended to a non-singular holomorphic foliation in a neighborhood of $H_{reg}$ in $M$, but in general, it is not possible to extend $\LL$ to a singular holomorphic foliation in a neighborhood of $H$. There are examples of singular Levi-flat hypersurfaces whose Levi foliations extend to singular \textit{holomorphic webs} in the ambient space, see for instance \cite{generic} and \cite{shafikov}. When there is a singular holomorphic foliation $\F$ in the ambient space $M$ that coincide with the Levi foliation on $H_{reg}$, we say either that $H$ is \textit{invariant} by $\F$ or that $\F$ is \textit{tangent} to $H$. Cerveau and Lins Neto \cite{alcides} proved that germs of singular foliations of codimension one at $(\C^N,0)$ tangent to real analytic Levi-flat hypersurfaces have meromorphic (possibly holomorphic) first integrals. We recall that a non-constant function $f$ is the \textit{first integral} for a foliation $\F$ if each leaf of $\F$ is contained in a level set of $f$. In the global context, the same problem has been studied in \cite{andres} and \cite{projective}.
\par The aim of this paper is to study holomorphic foliations tangent to real analytic Levi-flat subsets in complex manifolds. An irreducible real analytic subvariety $H\subset M$, where $M$ is an $N$-dimensional complex manifold, $N\geq 2$, is a \textit{Levi-flat subset} if it has real dimension $2n+1$ and its regular part $H_{reg}$ is foliated by immersed complex manifolds of complex dimension $n$. Similarly to the case of hypersurfaces, this foliation is called \textit{Levi foliation} of $H$ and will be denoted by $\LL$. The number $n$ is the \textit{Levi dimension} of $H$. We use the qualifier ``\emph{Levi}'' for the foliation, its leaves, and its dimension. Since we deal with real analytic Levi-flat subsets in complex manifolds we shall consider that $H$ is \textit{coherent}. Coherence implies that $H$ admits a global complexification \cite[p. 40]{guaraldo}. Here coherent means
that its ideal sheaf $\mathcal{I}(H)$ in $\mathcal{A}_{\R, M}$, the sheaf of germs of real analytic functions with real values in $M$, is a coherent sheaf of $\mathcal{A}_{\R, M}$-modules. It follows from Oka's theorem \cite[p. 94 Proposition 5]{Narasimhan} that $H$ is coherent if the sheaf $\mathcal{I}(H)$ is \textit{locally finitely generated}, the latter means that for every point $p\in H$ there exists an open neighborhood $U\subset M$ and a finite number of functions $\varphi_j$, real analytic in $U$ and vanishing on $H$, such that for any $q\in U$, the germs of $\varphi_j$ at $q$ generate the ideal $\mathcal{I}(H_q)$, where $H_q$ is the germ of $H$ at $q$. We remark that not every real analytic subset is coherent as we shall see in Section \ref{Levi-flat subset} of this paper. 
\par In \cite{brunella}, singular Levi-flat subsets appear in the result of the lifting of a real analytic Levi-flat hypersurface to the projectivized cotangent bundle of the ambient space through the Levi foliation and in \cite{dicritical}, the authors gave a complete characterization of dicritical singularities of local Levi-flat subsets in terms of their Segre varieties. 
\par Let $Y$ be a complex projective surface, $T\subset Y$ be a real analytic Levi-flat hypersurface, $X\subset\mathbb{P}^N$, $N\geq 3$, be a complex projective subvariety of complex dimension $k<N$  and 
$\rho:X\dashrightarrow  Y$ be a dominant rational map. Then it is easy to show that $H=\overline{\rho^{-1}(T)}$ is a real analytic Levi-flat subset in $\mathbb{P}^N$ and so $H$ is a Levi-flat subset defined via pull-back. Therefore, one natural question is:
\begin{center}
\textit{Given a real analytic Levi-flat subset $H\subset\mathbb{P}^N$. Under what condition, $H$ is given by the pull-back of a Levi-flat hypersurface in a projective complex surface via a rational map?}
\end{center}
\par In \cite{lebl}, Lebl gave sufficient conditions for a real analytic Levi-flat hypersurface in $\mathbb{P}^N$ to be a pull-back of a real algebraic curve in $\C$ via a meromorphic function. In \cite{bretas}, Bretas et al. proved an analogous result for real analytic Levi-flat subsets in $\mathbb{P}^N$. The main hypothesis in these articles is that the Levi foliation has infinitely many \textit{algebraic leaves}. In this paper, we give an answer to the question, assuming that $H$ is invariant by a singular holomorphic foliation on $\Pn^N$ with \textit{quasi-invariant subvarieties} (see 
Section \ref{quasi}). An irreducible complex subvariety $S\subset X$ of complex dimension $n$ is \textit{quasi-invariant} by a global $n$-dimensional foliation $\F$ on a complex projective manifold $X$ if it is not $\F$-invariant, but the restriction to the foliation $\F$ to $S$ is an \textit{algebraically integrable foliation} of dimension $n-1$, i.e. every leaf of $\F|_{S}$ is algebraic. 
The concept of \textit{quasi-invariant subvarieties} was introduced by Pereira-Spicer \cite{jorge} for codimension one holomorphic foliations on complex projective manifolds to prove a variant of the classical Darboux-Jouanolou Theorem. Here we shall use this concept for Levi foliations to prove our main result:
\begin{maintheorem}\label{Theorem_1}
Let $H\subset\mathbb{P}^N$, $N\geq 3$, be an irreducible real analytic Levi-flat subset of Levi dimension $n$ invariant by an $n$-dimensional singular holomorphic foliation $\F$ on $\mathbb{P}^N$. Suppose that $H$ is coherent and $n>N/2$. If the Levi foliation has infinitely many quasi-invariant subvarieties of complex dimension $n$, then there exists a unique projective subvariety $X$ of complex dimension $n+1$ containing $H$ such that either there exists a rational map $R:X\dashrightarrow\Pn^{1}$, and real algebraic curve $C\subset\Pn^{1}$ such that $\overline{H_{reg}}\subset\overline{R^{-1}(C)}$ or there exists a dominant rational map $\rho:X\dashrightarrow Y$ on a projective surface $Y$ and a semianalytic Levi-flat subset $T\subset Y$ such that $\overline{H_{reg}}\subset\overline{\rho^{-1}(T)}$.
\end{maintheorem}
\par We emphasize that the hypothesis $n>N/2$ implies that $H$ is necessarily a real analytic subvariety with singularities. In fact, Ni-Wolfson \cite[Theorem 2.4]{Ni} proved that no nonsingular real analytic Levi-flat subset of the Levi dimension $n$ exist in $\Pn^N$, $n>N/2$.

\par Applying Theorem \ref{Theorem_1} to $n=N-1$, we get the following corollary:
\begin{maincorollary}\label{corollary_1}
Let $H\subset\mathbb{P}^N$, $N\geq 3$, be an irreducible coherent real analytic Levi-flat hypersurface invariant by a codimension one holomorphic foliation $\F$ on $\mathbb{P}^N$. If the Levi foliation has infinitely many quasi-invariant complex hypersurfaces, then either there exists a rational map $R:\Pn^N\dashrightarrow\Pn^{1}$, and real algebraic curve $C\subset\Pn^{1}$ such that $\overline{H_{reg}}\subset\overline{R^{-1}(C)}$ or there exists a dominant rational map $\rho:\Pn^N\dashrightarrow Y$ on a projective surface $Y$ and a semianalytic Levi-flat subset $T\subset Y$ such that $\overline{H_{reg}}\subset\overline{\rho^{-1}(T)}$.
\end{maincorollary}
\par When $H$ is a real analytic hypersurface, the above corollary gives a nice characterization of coherent real analytic Levi-flat hypersurfaces in $\Pn^N$, $N\geq 3$, invariant by codimension one holomorphic foliations which admit infinitely many quasi-invariant complex hypersurfaces. Observe that, in order to improve our results, we need to extend the Levi foliation of a Levi-flat subset to a holomorphic foliation in the ambient space. Therefore, another interesting question is:
\begin{center}
\textit{Given a real analytic Levi-flat subset $H\subset\mathbb{P}^N$ with Levi foliation $\mathcal{L}$. Under what condition, $\mathcal{L}$ extend to a singular holomorphic foliation on $\Pn^N$?}
\end{center}
\par When $H$ is a local real analytic Levi-flat hypersurface, Lebl solved the above question in the non-dicritical case in \cite{singularlebl}. 
\par The paper is organized as follows: in Section \ref{quasi}, we define the concept of quasi-invariant subvarieties of a foliation with complex leaves and state the main result of \cite{jorge}, such a result is key to prove Theorem \ref{Theorem_1}. Section \ref{Levi-flat subset} is devoted to the study of real analytic Levi-flat subset in complex manifolds, using some results of \cite{brunella} and \cite{bretas}, we prove the algebraic extension of the intrinsic complexification of $H$. In Section \ref{teorema_sec}, we prove Theorem \ref{Theorem_1} and in Section \ref{coro_sec} we prove Corollary \ref{corollary_1}. Finally, in Section \ref{example}, we give two examples. The first is an example of a Levi-flat hypersurface where Theorem \ref{Theorem_1} applies. In the second example, 
we construct a Levi-flat hypersurface in $\Pn^3$ that is not a pull-back of a Levi-flat hypersurface of $\Pn^2$ under a rational map. Moreover, this example also is not a pull-back of a real algebraic curve under a meromorphic function.

\section{Foliations with complex leaves and quasi-invariant subvarieties}\label{quasi}
\subsection{Foliations with complex leaves}
A \textit{foliation with complex leaves of complex dimension $n$} is a smooth foliation $\mathcal{G}$ of dimension $2n$ whose local models are domains $U=W\times B$ of $\C^n\times\mathbb{R}^{k}$, $W\subset\C^n$, $B\subset\mathbb{R}^k$ and whose local transformations are of the form 
\begin{eqnarray}\label{coordinates}
\varphi(z,t)=(f(z,t),h(t)),
\end{eqnarray}
where $f$ is holomorphic with respect to $z$. A domain $U$ as above is said to be a \textit{distinguished coordinate domain} of $\mathcal{G}$ and $z=(z_1,\ldots,z_n)$, $t=(t_1,\ldots,t_k)$ are said to be \textit{distinguished local coordinates}. As examples of such foliations we have the Levi foliations of Levi-flat hypersurfaces of $\C^n$, see for instance \cite{sad} and \cite{levi}. 
\par If we replace $\mathbb{R}^k$ by $\C^k$ and in (\ref{coordinates}) we assume $t\in\C^k$ and that $f,h$ are holomorphic with respect to $z$, $t$ then we get the notion of \textit{holomorphic foliation} of \textit{complex codimension} $k$. 
\par Now we define foliations with singularities. Let $M$ be a complex manifold. A \textit{singular foliation with complex leaves $\mathcal{G}$ of dimension $n$} on $M$ is a foliation with complex leaves of dimension $n$ on $M\setminus E$, where $E$ is a real analytic subvariety of $M$ of real dimension $<2n$. A point $p\in E$ is called a \textit{removable singularity} of $\mathcal{G}$ of there is a chart $(U,\varphi)$ around $p$, compatible with the atlas $\mathcal{A}$ of $\mathcal{G}$ restricted to $M\setminus E$, in the sense that $\varphi\circ\varphi^{-1}_{i}$ and $\varphi_i\circ\varphi^{-1}$ have the form (\ref{coordinates}) for all $(U_i,\varphi_i)\in\mathcal{A}$ with $U\cap U_i\neq\emptyset$. The set of non-removable singularities of $\mathcal{G}$ in $E$ is called the \textit{singular set} of $\mathcal{G}$, and is denoted by $\sing(\mathcal{G})$.
\subsection{Quasi-invariant subvarieties}
\par Let $Z$ be a projective manifold of complex dimension $N\geq 2$ and let $\mathscr{G}$ be a foliation with complex leaves of dimension $n$ on $Z$. 
\begin{definition}
We say that $\mathscr{G}$ is an \textit{algebraically integrable foliation} on $Z$ if every leaf of $\mathscr{G}$ is algebraic, i.e. every leaf of $\mathscr{G}$ is a projective complex subvariety in $Z$.
\end{definition}
Motivated by \cite{jorge}, we define the concept of a \textit{subvariety quasi-invariant} by a real analytic foliation with complex leaves. 
\begin{definition}
An irreducible subvariety $S\subset Z$ of complex dimension $n$ is \textit{quasi-invariant} by a foliation $\mathscr{G}$ if it is not $\mathscr{G}$-invariant, but the restriction of the foliation $\mathscr{G}$ to $S$ is an algebraically integrable foliation. 
\end{definition}
We note that the restriction foliation $\mathscr{G}|_{S}$ is a codimension one foliation on $S$ and when $\mathscr{G}|_{S}$ is an algebraically integrable foliation, we have that every leaf of $\mathscr{G}|_{S}$ are projective complex hypersurfaces in $S$.
Codimension one holomorphic foliations on $Z$ which admit infinitely many quasi-invariant hypersurfaces have been studied in \cite{jorge} and its main result is the following.
\begin{theorem}[Pereira-Spicer \cite{jorge}]\label{teo_jorge}
Let $\F$ be a codimension one holomorphic foliation on a projective manifold $Z$. If $\F$ admits infinitely many quasi-invariant hypersurfaces then either $\F$ is an algebraically integrable foliation, or $\F$ is a pull-back of a foliation of dimension one on a projective surface under a dominant rational map. 
\end{theorem}

\section{Real analytic subsets}\label{Levi-flat subset}
\subsection{Coherent real analytic subsets.} We present some of the fundamental results concerning coherent real analytic subsets. 
\par Let $H$ be a real analytic subset in an open set $U\subset\mathbb{C}^{n}$ and let $\mathcal{I}(H)$ be its ideal sheaf, it is the sheaf of germs of real analytic functions with real values vanishing on $H$. 
\begin{definition}
$H$ is said to be coherent if $\mathcal{I}(H)$ is a coherent sheaf of $\mathcal{A}_{\R,U}$-modules, where $\mathcal{A}_{\R,U}$ is the sheaf of germs of real analytic functions with real values in $U$.
\end{definition}
\begin{proposition}\cite[p. 95]{Narasimhan}\label{nara}
If $H$ is a coherent real analytic subset and the germ $H_p$ of $H$ at $p$ is irreducible, then for $q$ near $p$, we have $$\dim_{\mathbb{R}} H_p=\dim_{\mathbb{R}} H_q.$$
\end{proposition}
\par It is well known that locally, a real analytic subset always admits a complexification (see for instance \cite[p. 40]{guaraldo}) and it is not true  for global real analytic subsets. It is shown in \cite[p. 54]{guaraldo} that the global complexification of a coherent real analytic subset in a complex manifold always exists.  
\begin{theorem}\cite[p. 54]{guaraldo}
A real analytic subset in a complex manifold is coherent if and only if it admits a global complexification.
\end{theorem}
\par Now we build an irreducible real analytic hypersurface in $\mathbb{P}^{3}$ which is not coherent. Let $[z_0:z_1:z_2:z_3]$ be the homogeneous coordinates in $\mathbb{P}^3$ and set $H\subset\mathbb{P}^3$ be the complex cone whose equation is
$$H=\{(z_3\bar{z}_0+\bar{z}_3 z_0)\left((z_1\bar{z}_0+\bar{z}_1 z_0)^2+(z_2 \bar{z}_0+\bar{z}_2 z_0)^2\right)-(z_1\bar{z}_0+\bar{z}_1 z_0)^3=0\}.$$
The germ $H_{p}$ of $H$ at $p=[1:0:0:0]$ is irreducible and of real dimension 5 at $p$. However, in a neighborhood of $[1:0:0:z]$, $z\neq 0$, $H$ reduces to the complex line $z_1=z_2=0$, which is of real dimension 2. By Proposition \ref{nara}, it follows that $H$ is not coherent. 
\subsection{Levi-flat subset in complex manifolds.} We give a brief resume of definitions and some known results about real analytic Levi-flat subsets in complex manifolds. 
Let $H$ be an irreducible real analytic Levi-flat subset of Levi dimension $n$ in an $N$-dimensional complex manifold $M$. The notion of
 Levi-flat subset germifies and, in general, we do not distinguish a germ at $(\C^{N},0)$ from its realization in some neighborhood $U$ of $0 \in \C^{N}$. 
If $p\in H_{reg}$ then, according to \cite[Proposition 3.1]{bretas}, there exists a holomorphic coordinate system $z=(z',z'')\in \C^{n+1} \times \C^{N-n-1}$ such that $z(p)=0\in\C^N$ and the germ of $H$ at $p$ is defined by
\begin{equation}
\label{levi-local-form}
H=\{z = (z',z'') \in \C^{n+1} \times \C^{N-n-1}:\  \im(z_{n+1})=0,\,\,\,\,z''=0\},
\end{equation}
where $z'=(z_1, ... ,z_{n+1})$ and $z''=(z_{n+2},...,z_N)$ and the Levi foliation is given by
$$\{z = (z',z'') \in \C^{n+1} \times \C^{N-n-1}:\  z_{n+1}=c,\,\,\,\, z''=0,  \ \text{with} \ c \in \mathbb{R}\}.$$ This trivial model is, in fact, a local form for a non-singular real analytic Levi-flat subset. Note that  in the local form \eqref{levi-local-form}, $\{z'' = 0\}$ corresponds to the unique local  $(n+1)-$dimensional complex subvariety of the ambient space containing the germ of $H_{reg}$ at $p$.  These local subvarieties glue together forming a complex variety defined in a whole neighborhood of $H_{reg}$.  It is analytically extendable  to a neighborhood of $\ov{H_{reg}}$ by the following theorem:

\begin{theorem}[Brunella \cite{brunella}]\label{Levi-viz-Hi}
Let $M$ be an $N-$dimensional complex manifold and $H \subset M$ be a  real analytic Levi-flat subset of Levi dimension $n$.  Then,   there exists  a neighborhood $ V \subset M$ of $\overline{H_{reg}}$   and  a unique complex    variety $X \subset V$ of dimension $n+1$ containing $H$.
\end{theorem}
The variety $X$ is the realization in the neighborhood $V$   of a germ of complex analytic
variety around $H$.
We  denote it --- or its germ --- by $H^{\imath}$ and call it  \emph{intrinsic complexification} or  \emph{$\imath$-complexification} of $H.$ It plays a central role in the theory of real analytic Levi-flat subsets. The notion of intrinsic complexification also appears in \cite{sukhov} with the name of the \emph{Segre envelope}. If $H$ is invariant by a holomorphic foliation on $M$, the same holds for its $\imath$-complexification, see for instance \cite[Proposition 3.3]{bretas}.
\begin{proposition}\label{fol-inv}
Let $H \subset M$ be a  real analytic Levi-flat subset of Levi dimension $n$, where
 $M$ is a complex manifold of dimension $N$. If $H$ is
  invariant by an $n$-dimensional  holomorphic foliation $\mathcal{F}$ on $M$, then its $\imath$-complexification $H^\imath$ is also invariant by $\mathcal{F}$.
\end{proposition}
As a consequence, if we denote by $\F^{\imath}:=\F|_{H^{\imath}}$ (the restriction of $\F$ to $H^{\imath}$), we have $\F^{\imath}$ has codimension one in $H^{\imath}$. The following proposition shows the importance of the assumption of the \textit{coherence} of a Levi-flat subset. 
\begin{proposition}\cite[Proposition 3.6]{bretas}\label{coherente}
Let $M$ be an $N$-dimensional complex manifold and $H\subset M$ be an irreducible real analytic Levi-flat subset of Levi dimension $n$. Suppose that $H$ is coherent. Then, there exist an open neighborhood $V\subset M$ of $H$ and a unique irreducible complex subvariety $X$ of $V$ of complex dimension $n+1$ containing $H$.
\end{proposition}
\par  The variety $X$  is the small variety of complex dimension $n+1$ that contains $H$. Again, let us denote this variety by $H^{\imath}$,  the intrinsic complexification of $H$.

\subsection{Levi-flat subsets in complex projective spaces}
In this subsection, we state some results of real analytic Levi-flat subset in $\Pn^N$. 
Let $\sigma:\C^{N+1}\to\Pn^N$ be the natural projection. Suppose that $H$ is a real-analytic subvariety of $\Pn^N$. Define the set $\tau(H)$ to be the set of points $z\in\C^{N+1}$ such that $\sigma(z)\in H$ or $z=0$. A real analytic subvariety $H\subset\Pn^{N}$ is said to be \textit{algebraic} if $H=\sigma(V)$ for some real algebraic complex cone $V$ in $\C^{N+1}$. A set $V$ is a complex cone when $p\in V$ implies $\lambda p \in V$ for all $\lambda\in\C$. 
\par The following construction offers several examples of Levi-flat subsets in $\Pn^N$.
\begin{proposition}\cite[Proposition 6.1]{bretas}
Let $X\subset\Pn^N$ be an irreducible $(n+1)$-dimensional algebraic variety, $R$ be a rational function in $X$ and $C\subset \Pn^1$ be a real algebraic one-dimensional subvariety. Then the set $\overline{R^{-1}(C)}$ is a real algebraic Levi-flat subset of Levi dimension $n$ whose $\imath$-complexification is $X$
\end{proposition}
\par When we add the hypothesis that the Levi-flat subset is invariant by a singular holomorphic foliation in the ambient space, we can state
a reciprocal result. 
\begin{proposition}\cite[Proposition 6.3]{bretas}\label{lebl_prop}
Let $\F$ be a singular holomorphic foliation in $\Pn^N$ tangent to a real analytic Levi-flat subset $H$ of Levi dimension $n$. Suppose that $H$ is coherent and its $\imath$-complexification extends to an algebraic subvariety $H^{\imath}$ in $\Pn^N$. If $\F^{\imath}$ has a rational first integral $R$, then there exists a real algebraic one-dimensional subvariety $C\subset\Pn^1$ such that $\overline{H_{reg}}\subset\overline{R^{-1}(C)}$.
\end{proposition}
Now, since $H$ is coherent, the intrinsic complexification $H^{\imath}$ is well-defined as a complex subvariety in a neighborhood of $H$. Our aim is to extend $H^{\imath}$ to an algebraic subvariety in $\Pn^N$. To get this, we use the following extension theorem.
\begin{theorem}[Chow \cite{chow}]\label{chow_theorem}
Let $Z\subset\Pn^N$ be a complex algebraic subvariety of dimension $k$ and $V$ be a connected neighborhood of $Z$ in $\Pn^{N}$. Then any complex analytic subvariety of dimension higher than $N-k$ in $V$ that intersects $Z$ extends algebraically to $\Pn^N$.
\end{theorem} 
Under certain hypotheses, we can prove that the $\imath$-complexification $H^{\imath}$ can be extended to $\Pn^{N}$.
\begin{proposition}\label{extension}
Let $H\subset\Pn^N$, $N\geq 3$, be an irreducible coherent real analytic Levi-flat subset of Levi dimension $n$ such that $n>N/2$. If the Levi foliation $\LL$ has a quasi-invariant complex algebraic subvariety of complex dimension $n$, then $H^{\imath}$ extends algebraically to $\Pn^N$.
\end{proposition}
\begin{proof}
Denote by $L$ such quasi-invariant algebraic complex subvariety with $\dim_{\C}L=n-1$. Since $L$ algebraic with $L\subset H^{\imath}$ and $\dim_{\C}H^{\imath}=n+1>N-(n-1)$, we can apply Theorem \ref{chow_theorem} to prove that $H^{\imath}$ extends algebraically to $\Pn^{N}$.
\end{proof}
\par To end this section, we shall prove the following proposition.
\begin{proposition}\label{singular_H}
Let $H\subset\Pn^N$ be an irreducible coherent real analytic Levi-flat subset of Levi dimension $n$ invariant by an $n$-dimensional singular holomorphic foliation $\F$ in $\Pn^{N}$. Suppose that the $\imath$-complexification $H^{\imath}$ extends to an algebraic variety in $\Pn^{N}$. If the Levi-foliation $\mathcal{L}$ has infinitely many quasi-invariant algebraic subvarieties of complex dimension $n-1$. Then, either the foliation $\F^{\iota}=\F|_{H^{\imath}}$ has a rational first integral in $H^{\imath}$, or   
$\F^{\iota}$ is a pull-back of a foliation on a projective surface under a dominant rational map.
\end{proposition}
\begin{proof}
First of all, we need to desingularize the $\imath$-complexification $H^{\imath}$. According to Hironaka desingularization theorem, there exist a complex manifold $\widetilde{H^{\imath}}$ and a proper bimeromorphic morphism $\pi:\widetilde{H^{\imath}}\to H^{\imath}$ such that 
\begin{enumerate}
\item $\pi:\widetilde{H^{\imath}}\setminus(\pi^{-1}(\sing(H^{\imath}))\to H^{\imath}\setminus \sing(H^{\imath})$ is a biholomorphism, 
\item $\pi^{-1}(\sing(H^{\imath}))$ is a simple normal crossing divisor.
\end{enumerate}
Since $H^{\imath}$ is compact then $\widetilde{H^{\imath}}$ is too. We lift $\F^{\imath}$ to an $n$-dimensional singular holomorphic foliation  $\widetilde{\F^{\imath}}$ on $\widetilde{H^{\imath}}$. Since $\dim_{\C}\widetilde{H^{\imath}}=n+1$, we have $\widetilde{\F^{\imath}}$ has codimension one on $\widetilde{H^{\imath}}$ and the tangency condition between $\F^{\imath}$ and $H$ implies that $\F^{\imath}$ has infinitely many quasi-invariant closed subvarieties (these are algebraic and of codimension one in $\widetilde{H^{\imath}}$). Thus the same holds for $\widetilde{\F^{\imath}}$. By Theorem \ref{teo_jorge}, either $\widetilde{\F^{\imath}}$ has a rational first integral or there exist a dominant rational map $\tilde{\rho}:\widetilde{H^{\imath}}\dashrightarrow Y$, where $Y$ is a projective complex surface, $\mathcal{G}$ is a foliation by curves on $Y$ and $\widetilde{\F^{\imath}}=\tilde{\rho}^{*}(\mathcal{G})$. If $\widetilde{\F^{\imath}}$ admits a rational first integral in $\widetilde{H^{\imath}}$, then all leaves of $\widetilde{\F^{\imath}}$ are compact and so their $\pi$-images are compact leaves of $\F^{\imath}$ in $H^{\imath}$. Applying G\'omez-Mont's theorem \cite{gomez}, we have that there exists a one-dimensional projective manifold $S$ and a rational map $f:H^{\imath}\dashrightarrow S$ whose fibers contain the leaves of $\F^{\imath}$. A rational first integral is obtained by composing $f$ with any non-constant rational map $r:S\dashrightarrow\mathbb{P}^{1}$. If $\widetilde{\F^{\imath}}$ is a pull-back of a foliation $\mathcal{G}$ on a projective complex surface $Y$ under a dominant rational map $\tilde{\rho}:\widetilde{H^{\imath}}\dashrightarrow Y$ then $\F^{\imath}$ is the pull-back of $\mathcal{G}$ under $\rho:=\tilde{\rho}\circ\pi^{-1}:H^{\imath}\dashrightarrow Y$, since $\pi$ is a birational map.
\end{proof}
\section{Proof of Theorem \ref{Theorem_1}}\label{teorema_sec}

\par With all the above results, we can prove Theorem \ref{Theorem_1}.
\begin{maintheorem}
Let $H\subset\mathbb{P}^N$, $N\geq 3$, be an irreducible real analytic Levi-flat subset of Levi dimension $n$ invariant by an $n$-dimensional singular holomorphic foliation $\F$ on $\mathbb{P}^N$. Suppose that $H$ is coherent and $n>N/2$. If the Levi foliation has infinitely many quasi-invariant subvarieties of complex dimension $n$, then there exists a unique projective subvariety $X$ of complex dimension $n+1$ containing $H$ such that either there exists a rational map $R:X\dashrightarrow\Pn^{1}$, and real algebraic curve $C\subset\Pn^{1}$ such that $\overline{H_{reg}}\subset\overline{R^{-1}(C)}$ or there exists a dominant rational map $\rho:X\dashrightarrow Y$ on a projective surface $Y$ and a semianalytic Levi-flat subset $T\subset Y$ such that $\overline{H_{reg}}\subset\overline{\rho^{-1}(T)}$.
\end{maintheorem}
\begin{proof}
By Proposition \ref{coherente}, there exist an open neighborhood $V\subset \Pn^{N}$ of $H$ and a unique irreducible complex subvariety $H^{\imath}$ of $V$ of complex dimension $n+1$ containing $H$. The Proposition \ref{fol-inv} implies that $H^{\imath}$ is invariant by $\F$ and moreover it extends algebraically to $\Pn^{N}$ by Proposition \ref{extension}.
We denote $\F^{\imath}:=\F|_{H^{\imath}}$ the restrict foliation to $H^{\imath}$. Observe now that $\F^{\imath}$ is a foliation of codimension one on $H^{\imath}$ which admit infinitely many quasi-invariant subvarieties of complex dimension $n-1$. Therefore, either $\F^{\iota}$ has a rational first integral in $H^{\imath}$, or   
$\F^{\iota}$ is a pull-back of a foliation on a projective surface under a dominant rational map by Proposition \ref{singular_H}.
\par If $\F^{\imath}$ has a first integral $R$ then there exists a real algebraic curve $C\subset\Pn^1$ such that $\overline{H_{reg}}\subset\overline{R^{-1}(C)}$ by Proposition \ref{lebl_prop}. Now if we assume that $\F^{\imath}$ is a pull-back of a foliation $\mathcal{G}$ on a projective complex surface $Y$ under a dominant rational map $\rho:H^{\imath}\dashrightarrow Y$. Then we can take $X=H^{\imath}$. Let us prove that there exists a semianalytic Levi-flat subset $T\subset Y$. Indeed, let $z\in H_{reg}\setminus Ind(\rho)$ (here $Ind(\rho)$ denotes the indeterminacy set of $\rho$). Then there exists a neighborhood $U\subset H^{\imath}\setminus Ind(\rho)$ of $z$ and a non-singular real analytic curve $\gamma:(-\epsilon,\epsilon)\to U$ such that $\gamma(0)=z$, $\{\gamma\}\subset H_{reg}$, and such that $\gamma$ is transverse to the Levi foliation $\LL$ on $H_{reg}$. Let $L_{\gamma(t)}$ be the leaf of $\LL$ through $\gamma(t)$.  
Since $L_{\gamma(t)}$ is also a leaf of $\F^{\imath}$ and $\F^{\imath}=\rho^{*}(\mathcal{G})$, then $\rho(L_{\gamma(t)})$ is a leaf of $\mathcal{G}$. Let us denote $A_t=\overline{\rho(L_{\gamma(t)})}\subset Y$ and define 
$$T_z:=\bigcup_{t\in(-\epsilon,\epsilon)}A_t\subset V_z,$$ where $V_z$ is a neighborhood of $T_z$ on $Y$.
Note that $T_z$ is a union of complex subvarieties parametrized by $t$ such that each $A_t$ contains leaves of $\mathcal{G}$, thus $T_z$ is a semianalytic Levi-flat subset on $V_z$.  These local constructions are sufficiently canonical to be patched together when $z$ varies on $H_{reg}$: if $T_{z_{1}}\subset V_{z_{1}}$ and $T_{z_{2}}\subset V_{z_{2}}$ are as above, with $V_{z_{1}}\cap V_{z_{2}}\neq \emptyset$, then $T_{z_{1}}\cap V_{z_{1}}\cap V_{z_{2}}$ and $T_{z_{2}}\cap V_{z_{1}}\cap V_{z_{2}}$ have some common leaves of $\mathcal{G}$ because $\mathcal{G}$ is a global foliation defined on $Y$, so $T_{z_{1}}$ and  $T_{z_{2}}$ can be glued by identifying these leaves. In this way, we get a semianalytic Levi-flat subset $T$ in $Y$. 
\par Finally, we assert that $\overline{H_{reg}}\subset \overline{\rho^{-1}(T)}$. In fact, let $w\in \overline{H_{reg}}$, then there exists a sequence $z_k\to w$, $z_k\in H_{reg}$, so $\rho(z_k)\in T$ which imply that $z_k\in \rho^{-1}(T)$ and $w\in \overline{\rho^{-1}(T)}$. This finishes the proof. 
\end{proof}
\section{Proof of Corollary \ref{corollary_1}}\label{coro_sec}
\begin{maincorollary}
Let $H\subset\mathbb{P}^N$, $N\geq 3$, be an irreducible coherent real analytic Levi-flat hypersurface invariant by a codimension one holomorphic foliation $\F$ on $\mathbb{P}^N$. If the Levi foliation has infinitely many quasi-invariant complex hypersurfaces, then either there exists a rational map $R:\Pn^N\dashrightarrow\Pn^{1}$, and real algebraic curve $C\subset\Pn^{1}$ such that $\overline{H_{reg}}\subset\overline{R^{-1}(C)}$ or there exists a dominant rational map $\rho:\Pn^N\dashrightarrow Y$ on a projective complex surface $Y$ and a semianalytic Levi-flat subset $T\subset Y$ such that $\overline{H_{reg}}\subset\overline{\rho^{-1}(T)}$.
\end{maincorollary}
\begin{proof}
If $H$ is an irreducible real analytic Levi-flat hypersurface in $\Pn^N$, $N\geq 3$, then the Levi dimension of $H$ is $N-1$. Moreover $$N-1>N/2\iff N>2.$$
Thus, we can apply Theorem \ref{Theorem_1} to $H$, so there exist a unique projective subvariety $X$ of complex dimension $N$ containing $H$ such that either there exists a rational map $R:X\dashrightarrow\C$, and real algebraic curve $C\subset\mathbb{C}$ such that $\overline{H_{reg}}\subset\overline{R^{-1}(C)}$ or there exists a dominant rational map $\rho:X\dashrightarrow Y$ on a projective complex surface $Y$ and a semianalytic Levi-flat subset $T\subset Y$ such that $\overline{H_{reg}}\subset\overline{\rho^{-1}(T)}$.
Since $X\subset\Pn^N$ has complex dimension $N$, we must have $X=\Pn^N$ and hence we conclude the proof. 
\end{proof}
\section{Examples}\label{example}
\begin{example}
We give an example of a real analytic Levi-flat hypersurface in $\Pn^3$
where Theorem \ref{Theorem_1} applies. Let $$H=\{[z_0:z_1:z_2:z_3]\in\Pn^3:z_0z_1\bar{z}_2\bar{z}_3-z_2z_3\bar{z}_0\bar{z}_1=0\},$$
then $H$ is Levi-flat because it is foliated by the complex hypersurfaces
\begin{eqnarray}\label{eq1}
z_0z_1=cz_2z_3,\quad\text{where}\quad c\in\mathbb{R}.
\end{eqnarray}
Let $\F$ be the codimension one holomorphic foliation on $\Pn^3$ of degree two defined by $$\omega=z_1z_2z_3dz_0+z_0z_2z_3dz_1-z_0z_1z_3dz_2-z_0z_1z_2dz_3,$$
then $\F$ has a rational first integral $R:\Pn^3\dashrightarrow \Pn^1$ given by $$R[z_0:z_1:z_2:z_3]=[z_0z_1:z_2z_3].$$
Since the leaves of $\F|_H$ coincide with the leaves of the Levi foliation (\ref{eq1}), $H$ must be invariant by $\F$. On the other hand, note that $H=\overline{R^{-1}(C)}$, where $$C=\{[t:u]\in\Pn^1:t\bar{u}-u\bar{t}=0\}.$$ 
\end{example}
\begin{example}
In the following example, we construct a real analytic Levi-flat hypersurface $H$ in $\Pn^3$ that is not a pull-back of a Levi-flat hypersurface of $\Pn^2$ under a rational map, furthermore, $H$ also is not a pull-back of a real algebraic curve under a meromorphic function.
\par Consider $z=(z_0,z_1,z_2,z_3)$,  $\bar{z}=(\bar{z}_0,\bar{z}_1,\bar{z}_2,\bar{z}_3)$ and
\[F(z,\bar{z})=\det\left( \begin{array}{cccccc}
z_{0} & z_{1} & z_{2} & z_3 & 0 & 0 \\
0 & z_{0} & z_{1} & z_{2} & z_3 & 0\\
0 & 0 & z_0 & z_1 & z_2 & z_3 \\
\bar{z}_{0} & \bar{z}_{1} & \bar{z}_{2} & \bar{z}_3 & 0 & 0 \\
0 & \bar{z}_{0} & \bar{z}_{1} & \bar{z}_{2} & \bar{z}_3 & 0\\
0 & 0 & \bar{z}_{0} & \bar{z}_{1} & \bar{z}_{2} & \bar{z}_3
\end{array} \right)\]
Define $H=\{[z_0:z_1:z_2:z_3]\in\Pn^3: F(z,\bar{z})=0\}$, $H$ is a real analytic hypersurface well defined since $F$ is a bihomogeneous polynomial of bi-degree $(3,3)$. Moreover, $H$ is Levi-flat, because it is foliated by the complex hyperplanes
\begin{eqnarray}\label{folhas}
z_0+cz_1+c^2 z_2+c^3 z_3=0,\quad\text{where}\quad c\in\mathbb{R}.
\end{eqnarray}
\par Let $\mathscr{W}$ be the codimension one holomorphic $3$-web on $\Pn^3$ given by the implicit differential equation $\Omega = 0$, 
\[\Omega=\det\left( \begin{array}{cccccc}
z_{0} & z_{1} & z_{2} & z_3 & 0 & 0 \\
0 & z_{0} & z_{1} & z_{2} & z_3 & 0\\
0 & 0 & z_0 & z_1 & z_2 & z_3 \\
dz_{0} & dz_{1} & dz_{2} & dz_3 & 0 & 0 \\
0 & dz_{0} & dz_{1} & dz_{2} & dz_3 & 0\\
0 & 0 & dz_{0} & dz_{1} & dz_{2} & dz_3
\end{array} \right)\]
Since the leaves of $\mathscr{W}|_{H}$ and $\LL$ are the same, we get $H$ is invariant by $\mathscr{W}$. 
\par Now, we prove that $H$ is not a pull-back of a Levi-flat hypersurface of $\Pn^2$. To prove this fact, we use the following result of \cite[Proposition 4.4]{junca}:
\begin{proposition}\label{junca_1}
Let $\omega_1$, $\omega_2$ and $\omega_3$ be independent germs of integrable 1-forms at $(\C^3,0)$ with singular sets of codimension at least two. Suppose that there exists
a non-zero holomorphic 2-form $\eta$, locally decomposable outside its singular set, that is tangent to each $\omega_i$, for $i=1,2,3$. Then $\omega_1$, $\omega_2$ and $\omega_3$ define foliations that are in a pencil. Furthermore, $\eta$ is integrable, defining the axis foliation of this pencil.
\end{proposition}
\par Suppose by contradiction that $H$ is a pull-back of a Levi-flat hypersurface under a dominant rational map $\rho:\Pn^3\dashrightarrow\Pn^2$. Then pick a point $p\in U_0$, where $U_0$ is an open subset in $\Pn^3$ such that $\rho|_{U_0}:U_0\subset\C^3\to\C^2$ is a holomorphic submersion. We may have needed to perhaps move to yet another point $p'\in U_0$ such that $U_0$ does not intersect the discriminant set of the web $\mathscr{W}$. We set $p=p'$ and works in a neighborhood of $U_0$. Therefore, the germ of $\mathscr{W}$ at $p$ is a decomposable 3-web, defined by the superposition of three independent foliations $\F_1$, $\F_2$, and $\F_3$. We can assume that these foliations are defined by independent germs of integrable 1-forms $\omega_1$, $\omega_2$, and $\omega_3$ respectively. Since $H$ is given by a pull-back, all the leaves of $\LL$ and, hence the leaves of $\mathscr{W}$ in $H\cap U_0$ are tangent to the fibers of $\rho|_{U_{0}}$, these fibers define a non-zero holomorphic 2-form $\eta_{\rho}$ that is tangent to each $\omega_i$, for $i=1,2,3$. Then, according to Proposition \ref{junca_1}, $\omega_1$, $\omega_2$, and $\omega_3$ define foliations that are in a pencil, an absurd. Hence, the assertion is proved.
\par Now we assert that $H$ is not a pull-back of a real algebraic curve under a meromorphic function. In fact, $H$ is a Levi-flat hypersurface in $\Pn^3$ such that there does not exist a point contained in infinitely many leaves of $\LL$, because, the leaves of $\LL$ are given by the equation (\ref{folhas}) and through at a point only pass three leaves. If $H$ is defined by a pull-back of a  meromorphic function, there has to exist a point $p$ of indeterminacy since the dimension is at least 2. Then through at $p$ pass infinitely many leaves of $\LL$. Since $H$ does not satisfy this property, we finish the proof of the assertion.
 
\end{example}

\vspace{1cm}
\noindent {\bf Acknowledgments.}
 The authors wish to express his gratitude to Maria Aparecida Soares Ruas (ICMC - USP, S\~ao Carlos) and Judith Brinkschulte (Universit\"{a}t Leipzig) for many valuable conversations and suggestions.


\begin{thebibliography}{99}

\bibitem{andres}
 Beltr\'an, A., Fern\'andez-P\'erez, A., and Neciosup, H.: Existence of dicritical singularities of Levi-flat hypersurfaces and holomorphic foliations. Geometriae Dedicata (2017), doi.org/10.1007/s10711-017-0303-4 

\bibitem{bretas}
Bretas, J., Fern\'andez-P\'erez, A., and Mol, R.: Holomorphic foliations tangent to Levi-flat subsets. J. Geom. Anal. (2019) 29: 1407. https://doi.org/10.1007/s12220-018-0043-1
\bibitem{brunella} 
Brunella, M.: Singular Levi-flat hypersurfaces and codimension one foliations. Ann. Sc. Norm. Super. Pisa Cl. Sci. (5) vol. VI, no. 4 $(2007)$, 661-672.
\bibitem{cartan}
Cartan, E.: Sur la g\'eom\'etrie pseudo-conforme des hypersurfaces de l'espace de deux variables complexes. Ann. Mat. Pura Appl., 11 (1) (1933), 17-90.
\bibitem{sad}
Cerveau, D., and Sad, P.: Fonctions et feuilletages Levi-flat. 
\`etude locale. Ann. Sc. Norm. Super. Pisa Cl. Sci. (5), 3(2):427-445, (2004).
\bibitem{alcides}  Cerveau, D., and Lins Neto, A.:
   Local Levi-flat hypersurfaces invariants by a codimension one holomorphic foliation. American Journal of Mathematics, vol. 133 no. 3, $(2011)$, 677-716. doi.org/10.1353/ajm.2011.0018
\bibitem{chow}
Chow, W.L.: On meromorphic maps of algebraic varieties. Ann. Math. 2 (89), 391-403 (1969).
\bibitem{generic} Fern\'andez-P\'erez, A.: On Levi-flat hypersurfaces with generic real singular set. J. Geom. Anal. (2013) 23: 2020. doi.org/10.1007/s12220-012-9317-1
\bibitem{projective}
Fern\'andez-P\'erez, A.: Levi-flat hypersurfaces tangent to projective foliations. J Geom. Anal. (2014) 24: 1959. doi.org/10.1007/s12220-013-9404-y
\bibitem{levi}
Fern\'andez-P\'erez, A., Mol, R., and Rosas, R.: On singular real analytic Levi-flat foliations. To be
published: The Asian Journal of Mathematics, 2021. 
\bibitem{guaraldo}
Guaraldo, F., Macr\`i P., and Tancredi, A.: Topics on real analytic spaces. Springer-Verlag, (2013).
\bibitem{gomez}
G\'omez-Mont, X.: Integrals for holomorphic foliations with singularities having all leaves compact. Ann. Inst. Fourier 39(2), 451-458 (1989).
\bibitem{junca}
Junca, D., and Mol, R.: Holomorphic vector fields tangent to foliations in dimension three. To be
published: An. Acad. Brasil. Ci\^enc., 2021.
\bibitem{lebl} Lebl, J.: Algebraic Levi-flat hypervarieties in complex projective space. J. Geom. Anal. (2012) 22: 410. https://doi.org/10.1007/s12220-010-9201-9
\bibitem{singularlebl} Lebl, J.: Singular set of a Levi-flat hypersurface is Levi-flat. Math. Ann. (2013) 355: 1177. doi.org/10.1007/s00208-012-0821-1
\bibitem{grass_lebl}
Lebl, J.: Singular Levi-flat hypersurfaces in complex projective space induced by curves in the Grassmannian. Internat. J. Math. 26 (2015), no. 5, 1550036, 17 pp. 
\bibitem{Narasimhan}
Narasimhan, R.: Introduction to the theory of analytic spaces. Lectures Notes in Mathematics. Springer-Verlag Berlin Heidelberg (1966). doi.org/10.1007/BFb0077071
\bibitem{Ni}
Ni, L., and Wolfson, J.: The Lefschetz theorem for CR submanifolds and the nonexistence of real analytic Levi flat submanifolds. Communications in Analysis and Geometry. Volume 11, Number 3, 553-564, (2003). 
\bibitem{jorge}
Pereira, J.V., and Spicer C.: Hypersurfaces quasi-invariant by codimension one foliations. Mathematishe Annalen (2019). doi 10.1007/s00208-019-01833-4

\bibitem{dicritical}
Pinchuk, S., Shafikov, R., and Sukhov, A.: On dicritical singularities of Levi-flat sets. Ark. Mat., 56 (2018), 395-408. DOI: 10.4310/ARKIV.2018.v56.n2.a12
\bibitem{shafikov}
 Shafikov, R., and Sukhov, A.: Germs of singular Levi-flat hypersurfaces and holomorphic foliations. Comment. Math. Helv. 90 (2015), 479-502. 
\bibitem{sukhov}
Sukhov, A.: Levi-flat world: a survey of local theory. Ufimsk. Mat. Zh., (2017), Volume 9, Issue 3, 172-185. 


\end{thebibliography}
\end{document}